\documentclass[11pt]{amsart}

\usepackage[T1]{fontenc}
\usepackage[utf8]{inputenc}
\usepackage{lmodern}
\usepackage[margin=1in]{geometry}
\usepackage{microtype}
\usepackage{amsmath,amssymb,mathtools}
\usepackage{enumitem}
\usepackage{booktabs}
\usepackage[dvipsnames]{xcolor}
\usepackage[colorlinks=true,linkcolor=MidnightBlue,citecolor=ForestGreen,urlcolor=MidnightBlue]{hyperref}
\usepackage{tikz}
\usetikzlibrary{arrows.meta}

\newtheorem{theorem}{Theorem}[section]

\newtheorem{corollary}[theorem]{Corollary}
\newtheorem{lemma}[theorem]{Lemma}
\theoremstyle{definition}

\theoremstyle{remark}

\newcommand{\Nzero}{\mathbb N_0}
\newcommand{\Z}{\mathbb Z}
\newcommand{\A}{\mathcal A}

\setlist[enumerate]{itemsep=2pt,topsep=4pt}
\allowdisplaybreaks

\title[The Multiplicative Persistence Conjecture]{The Multiplicative Persistence Conjecture: Resolving the \(2\)-Adic Obstruction for Nonzero Even Targets}

\author[P. Nyadjo Fonga]{Patrick Nyadjo Fonga}
\address{Department of Mathematics and Statistics,
Bowling Green State University,
Bowling Green, Ohio 43403, USA}

\email{nyadjop@bgsu.edu}

\begin{document}

\begin{abstract}
The multiplicative persistence problem studies the process of repeatedly
replacing a positive integer by the product of its digits until a
single digit, called the terminal digit, is reached. The classical conjecture
asserts that no decimal integer requires more than \(11\) iterations. Brier,
Clavier, Gutsche, and Naccache proved the conjecture for all odd terminal
digits. In
their approach to nonzero even terminal digits, they were led to infinite families of
decimal integers in which the numbers of digits \(2,\ldots,9\) are fixed,
while arbitrarily many digits \(1\) may be inserted. They conjectured that,
despite the infinitude of such a family, the exponent of \(2\) dividing its
elements is uniformly bounded. We prove this conjecture and obtain an explicit
bound depending only on the prescribed digit multiplicities. More generally,
our argument applies in every base \(b\geq3\) and to every prime \(p\mid b\).
In base \(10\), combining our bound with the method of Brier, Clavier, Gutsche, and Naccache yields a finite algorithm for a further analysis of each nonzero even terminal digit.
\medskip

\end{abstract}

\maketitle

\section{Introduction}

Throughout the present paper, if \(p\) is a prime and \(m\) is a nonzero integer, we
denote by \(v_p(m)\) the \(p\)-adic valuation of \(m\). We also set $v_p(0)=+\infty$ and write
$\Nzero=\{0,1,2,\ldots\}$. Let $b\ge2$ and let
\[
 n=\sum_{j=0}^{L-1} d_j b^j,
 \qquad d_j\in\{0,1,\ldots,b-1\},
 \qquad d_{L-1}\ne0,
\]
be the base-$b$ expansion of a positive integer $n$. We define the \emph{digit product} of \(n\) as
\[
 S_b(n):=\prod_{j=0}^{L-1} d_j.
\]
Starting from $n_0=n$, define $n_{k+1}=S_b(n_k)$. Whenever $n_k\ge b$, we have that
$S_b(n_k)<n_k$, so the sequence \((n_k)\) is decreasing and therefore eventually reaches a one-digit number in base \(b\). Therefore, we define the \emph{multiplicative persistence} or simply \emph{persistence} of $n$ in base $b$ as
\[
 P_b(n):=\min\{r\ge0:S_b^r(n)<b\}
\]
and we call
\[
 \Delta_b(n):=S_b^{P_b(n)}(n)\in\{0,1,\ldots,b-1\}
\]
the \emph{target digit} or \emph{terminal digit} or  simply \emph{terminal} of the integer \(n\) in base \(b\). 

Sloane introduced the notion of multiplicative persistence in 1973 and asked whether, for
each fixed base, the persistence is uniformly bounded \cite{Sloane1973}. In
base 10, this leads to the following well-known conjecture:

\noindent\textbf{Conjecture A.} For all \(n\in\mathbb N\), we have \(P_{10}(n)\le 11\).

The smallest known integer with multiplicative persistence \(11\) is
\(
277777788888899,
\)
and no example of larger persistence is currently known; see
\cite{PerezStyer2015}.

Perez and Styer performed interesting
computations in bases $2$ through $12$ and in the particular case of base 10,  they verified the bound $11$ given by Conjecture A
for all integers below $10^{1500}$ \cite{PerezStyer2015}. Further background is
available in \cite{BonuccelliColucciFaria2020, deFariaTresser2014, Guy2004,
LamontSmith2021}.

A useful way to study the problem is to work backward from the terminal digit.
Brier, Clavier, Gutsche, and Naccache \cite{BrierEtAl2021} developed this
approach in base 10. They showed that Conjecture A holds for odd
terminal digits: indeed, they proved that terminal digits $1,3,7,$ and $9$ lead to persistence at most
$1$, while terminal digit $5$ leads to persistence at most $5$. However, their approach to the even terminal digits \(2,4,6,\) and \(8\) leads
to exponential Diophantine equations whose analysis gives rise to the following conjecture
\cite[Conjecture~2]{BrierEtAl2021};

\medskip
\noindent\textbf{Conjecture B.}
Fix nonnegative integers \(\nu_2,\ldots,\nu_9\). Consider the family of
positive decimal integers containing exactly \(\nu_d\) occurrences of the
digit \(d\), for each \(d=2,\ldots,9\), an arbitrary number of occurrences
of the digit \(1\), and no occurrence of the digit \(0\). Then there exists
a constant
\(
C
\)
such that
\(
v_2(n)\le C
\)
for every integer \(n\) in this family.
\medskip

In this paper, we provide a proof for Conjecture~B; see Corollary~\ref{cor:brier-conjecture}.
In fact, we obtain a more general result that applies to arbitrary bases; see Theorem~\ref{thm:general-bound}

The paper is organized as follows. In Section~\ref{sec:digit-families}, we
establish an explicit uniform bound in every base \(b\ge 3\) for every prime
\(p\mid b\); see Theorem~\ref{thm:general-bound}. Since the bound found in Theorem~\ref{thm:general-bound} is not optimal, Section~\ref{sec:sharp} develops a finite recursive
procedure, given in Theorem~\ref{thm:exact-search}, for computing the exact
maximal valuation associated with a prescribed multiset of digits. Finally,
Section~\ref{sec:even-program} explains how these results apply to the
even-terminal-digit approach to multiplicative persistence.

\section{Preliminary results}\label{sec:digit-families}

We begin by giving three elementary properties of the digit-product map defined in the previous section.
Let \(b\geq 3\) and \(n\in\mathbb{N}\).

\begin{itemize}
    \item If the base-\(b\) expansion of \(n\) contains a zero digit, then
    \(
    S_b(n)=0.
    \)

    \item Inserting or removing digits equal to \(1\) in the base-\(b\) expansion of \(n\) does not change the value
    of \(S_b(n)\).

    \item Reordering the digits in the base-\(b\) expansion of \(n\) does not
    change the value of \(S_b(n)\).
\end{itemize}

Consequently, once a multiset of nonzero digits has been
fixed, the only way to obtain infinitely many integers with the same first
digit product is by inserting arbitrarily many digits \(1\) in arbitrary
positions. This motivates the following definition.

Let
\(
\nu=(\nu_2,\nu_3,\ldots,\nu_{b-1})\in\Nzero^{b-2}.
\)
We denote by \(\A_b(\nu)\) the family of positive integers whose base-\(b\)
expansions contain exactly \(\nu_d\) occurrences of the digit \(d\), for each
\(2\le d\le b-1\), an arbitrary number of occurrences of the digit \(1\), and
no occurrence of the digit \(0\). The digits \(2,\ldots,b-1\) will be referred
to as the \emph{exceptional digits} of the family. We write
\[
q=q(\nu):=\sum_{d=2}^{b-1}\nu_d
\]
for the total number of exceptional digits in every element of
\(\A_b(\nu)\).

\begin{lemma}\label{eq:repunit-correction}
Let \(x\in\A_b(\nu)\) have \(L\) digits in base \(b\). Denote its exceptional
digits by \(d_1,\ldots,d_q\), and let \(r_i\) be the position occupied by
\(d_i\), where
\(
r_i\in\{0,1,\ldots,L-1\},
\)
and \(r_1,\ldots,r_q\) are pairwise distinct. Then
\begin{equation*}
x=\frac{b^L-1}{b-1}
  +\sum_{i=1}^{q}(d_i-1)b^{r_i}.
\end{equation*}
\end{lemma}

\begin{proof}
The base-\(b\) integer consisting of \(L\) digits, all equal to \(1\), is
\[
1+b+b^2+\cdots+b^{L-1}=\frac{b^L-1}{b-1}.
\]
Here position \(0\) is the units position, position \(1\) is the coefficient
of \(b\), and, more generally, position \(r\) is the coefficient of \(b^r\).
Replacing the digit \(1\) in position \(r_i\) by \(d_i\) increases the value
of the integer by \((d_i-1)b^{r_i}\). Performing this replacement at every
exceptional position therefore gives
\[
x=\frac{b^L-1}{b-1}
  +\sum_{i=1}^{q}(d_i-1)b^{r_i},
\]
as claimed.
\end{proof}

Lemma~\ref{eq:repunit-correction} shows that, once the multiplicities
\(\nu_d\) are fixed, the exceptional digits themselves are fixed as a
multiset. Thus, the only quantities that vary among the elements of
\(\A_b(\nu)\) are the number \(L\) of base-\(b\) digits and the pairwise
distinct positions \(r_1,\ldots,r_q\). Lemma 
\ref{eq:repunit-correction} therefore gives a convenient representation of the entire family
\(\A_b(\nu)\).

The family \(\A_{10}(\nu)\) is exactly the family that appears in the study of
multiplicative persistence. Indeed, suppose in base 10 that
\[
S_{10}^2(x)=s
\]
and that one representation of \(s\) as a product of decimal digits is
\[
s=\prod_{d=2}^9 d^{\nu_d}.
\]
Then the intermediate value \(S_{10}(x)\) must have exactly \(\nu_d\)
occurrences of the digit \(d\), for \(2\le d\le 9\), together with an
arbitrary number of digits \(1\). Hence
\(
S_{10}(x)\in\A_{10}(\nu).
\)

At the same time, since \(S_{10}(x)\) is itself a product of decimal digits,
its prime divisors can only be \(2,3,5,\) and \(7\). Thus
\[
S_{10}(x)=2^t3^u5^v7^w,
\qquad t,u,v,w\in\Nzero.
\]
In particular,
\(
t=v_2\bigl(S_{10}(x)\bigr).
\)
Therefore, proving that \(v_2(y)\) is uniformly bounded for
\(y\in\A_{10}(\nu)\) is precisely what is needed to show that the exponent
\(t\) cannot grow arbitrarily large when the exceptional digits are fixed.
This is the content of Conjecture~B.

We now provide in the following lemma the elementary $p$-adic mechanism behind all our bounds.

\begin{lemma}\label{lem:valuation-control}
Let \(p\) be a prime and let \(b\geq2\) satisfy \(p\mid b\). Set
\(e:=v_p(b)\).
\begin{enumerate}[label=\textup{(\roman*)},leftmargin=1.5em]
\item If \(S,c\in\Z\setminus\{0\}\) and \(a\in\Nzero\) satisfy
\[
ea+v_p(c)>v_p(S),
\]
then
\[
v_p(S+cb^a)=v_p(S).
\]

\item Suppose that \(a,c,B,T\) are positive integers and
\[
T+cb^a<Bb^a.
\]
Then either
\begin{equation}\label{eq:one-step-bound-1}
v_p(T+cb^a)=v_p(T),
\end{equation}
or
\begin{equation}\label{eq:one-step-bound}
v_p(T+cb^a)
<
\log_p B+\frac{\log_p b}{e}\,v_p(T).
\end{equation}
\end{enumerate}
\end{lemma}

\begin{proof}
For \textup{(i)}, we have
\[
v_p(cb^a)=v_p(c)+ea>v_p(S).
\]
Since the two summands have distinct \(p\)-adic valuations,
\[
v_p(S+cb^a)
=\min\{v_p(S),v_p(cb^a)\}
=v_p(S).
\]

For \textup{(ii)}, if \(ea>v_p(T)\), then
\[
ea+v_p(c)>v_p(T),
\]
so part \textup{(i)} gives \eqref{eq:one-step-bound-1}. Otherwise,
\(ea\leq v_p(T)\), and hence
\[
a\leq\frac{v_p(T)}{e}.
\]
Therefore,
\[
p^{v_p(T+cb^a)}
\leq T+cb^a
<Bb^a
\leq Bb^{v_p(T)/e}.
\]
Taking logarithms to base \(p\) gives \eqref{eq:one-step-bound}.
\end{proof}

Lemma~\ref{lem:valuation-control} states that once the position \(a\) is large
enough that \(ea\) exceeds the current \(p\)-adic valuation, placing a new
digit at position \(a\) cannot change that valuation. When \(ea\) does not
exceed the current valuation, the new digit may affect it, but part~(ii)
provides an explicit upper bound for the resulting valuation.

\section{Uniform bounds for \(p\)-adic valuations}

Fix $b\ge3$. In this section, we show that once the number of occurrences of each digit
\(2,\ldots,b-1\) is fixed, there is a uniform limit on the power of any prime
dividing \(b\) that can divide the resulting integers. We also provide an
explicit formula for this limit. 

\begin{theorem}\label{thm:general-bound}
Let \(b\ge3\), let \(p\mid b\) be prime, and fix
\(
\nu=(\nu_2,\ldots,\nu_{b-1})\in\Nzero^{b-2}.
\)
Then
\[
\sup_{x\in\A_b(\nu)}v_p(x)<\infty.
\]
Moreover, if we set  \[B(\nu):=1+(b-1)\sum_{d=2}^{b-1}(d-1)\nu_d, \quad e:=v_p(b), \ \  \ \ \ \lambda:=\frac{\log_p b}{e}\]  and \[ \rho_p(\nu):=
\max\left(
\{0\}\cup
\left\{
v_p\bigl((b-1)(d-1)-1\bigr):
\begin{array}{l}
2\leq d\leq b-1, \ \
\nu_d>0,\quad p\mid d
\end{array}
\right\}
\right),\] then;\\
if none of the exceptional digits is divisible by \(p\), then
\[
v_p(x)=0
\qquad\text{for every }x\in\A_b(\nu).
\]
Otherwise, with \(q=\sum_{d=2}^{b-1}\nu_d\), every \(x\in\A_b(\nu)\) satisfies, when
\(\lambda>1\),
\begin{equation}\label{eq:general-bound}
\begin{split}
v_p(x)\le
\Biggl\lfloor
\rho_p(\nu)\lambda^q +\log_p B(\nu)\,
\frac{\lambda^q-1}{\lambda-1}
\Biggr\rfloor,
\end{split}
\end{equation}
while, when \(\lambda=1\),
\begin{equation}\label{eq:general-bound-pure-power}
v_p(x)\le
\left\lfloor
\rho_p(\nu)+q\log_p B(\nu)
\right\rfloor.
\end{equation}
\end{theorem}

\begin{proof}
Fix \(x\in\A_b(\nu)\).
Suppose first that none of the exceptional digits of \(x\) is divisible by
\(p\). Every digit of \(x\), including its units digit, is either
\(1\) or one of the exceptional digits. Since neither \(1\) nor
any exceptional digit is divisible by \(p\), the units digit is
not divisible by \(p\). As \(p\mid b\), it follows that \(p\nmid x\), and
therefore
\(
v_p(x)=0.
\)

We now assume that at least one prescribed exceptional digit of \(x\) is divisible by
\(p\). If the units digit of \(x\) is not divisible
by \(p\), then again \(v_p(x)=0\), so there is nothing to prove. We may
therefore restrict our attention to the case \(p\mid x\). In this case the
units digit must be an exceptional digit \(t\) satisfying \(p\mid t\).

Suppose that \(x\) has \(L\) digits in base
\(b\) with its \(q\) exceptional digits being
\(
d_1,\ldots,d_q,
\)
and that these digits occur at the pairwise distinct positions
\(
r_1,\ldots,r_q\in\{0,1,\ldots,L-1\},
\)
so that \(d_i\) is the digit occurring in position \(r_i\), with position \(0\)
corresponding to the units digit. We have from Lemma \ref{eq:repunit-correction} that
\begin{equation}\label{eq:scaled-expansion}
(b-1)x
=
-1+b^L+
\sum_{i=1}^{q}(b-1)(d_i-1)b^{r_i}.
\end{equation}
Apart from the initial term \(-1\), the right-hand side of \eqref{eq:scaled-expansion} contains exactly
\(q+1\) positive terms. We now arrange these \(q+1\) positive terms according to increasing powers of
\(b\). Since the positions \(r_1,\ldots,r_q\) are pairwise distinct and all
belong to \(\{0,\ldots,L-1\}\), while the additional exponent is \(L\), then we may write their exponents as
\[
0=a_1<a_2<\cdots<a_q<a_{q+1}=L,
\]
so that
\begin{equation}\label{eq:ordered-expansion}
(b-1)x=-1+\sum_{j=1}^{q+1}c_jb^{a_j},
\end{equation}
with
\[
\sum_{j=1}^{q+1}c_j
=
1+(b-1)\sum_{i=1}^{q}(d_i-1).
\]
Since the digit \(d\) occurs exactly \(\nu_d\) times among
\(d_1,\ldots,d_q\), we have
\[
\sum_{i=1}^{q}(d_i-1)
=
\sum_{d=2}^{b-1}(d-1)\nu_d.
\]
Therefore
\begin{equation}\label{eq:coefficient-sum}
\sum_{j=1}^{q+1}c_j
=
1+(b-1)\sum_{d=2}^{b-1}(d-1)\nu_d
=
B(\nu).
\end{equation}
For
\(1\le j\le q+1\), define
\[
T_j:=-1+\sum_{i=1}^{j}c_i b^{a_i},
\qquad
r_j:=v_p(T_j).
\]
Notice that
\[
T_1=-1+c_1=-1+(b-1)(t-1)>0.
\]

Because the units digit \(t\) is divisible by \(p\), the definition of
\(\rho_p(\nu)\) gives
\begin{equation} \label{eq1}
r_1
=
v_p\bigl((b-1)(t-1)-1\bigr)
\le \rho_p(\nu).
\end{equation}

Since \(p^e\mid b\), we have \(b\ge p^e\), and hence
\(
\lambda\ge1.
\)
We define a sequence of integers \(R_0,\ldots,R_q\) recursively by
\begin{equation*}\label{eq:recursive-R}
R_0:=\rho_p(\nu),
\quad
R_{k+1}:=
\left\lfloor
\lambda R_k+\log_p B(\nu)
\right\rfloor,
\quad 0\le k<q.
\end{equation*}
We shall prove that
\begin{equation}\label{eq:induction-claim}
r_j\le R_{j-1},
\qquad 1\le j\le q+1.
\end{equation}

For \(j=1\), this follows immediately from \eqref{eq1}.
Now let \(2\le j\le q+1\), and suppose inductively that
\[
r_{j-1}\le R_{j-2}.
\]
Since
\(
T_j=T_{j-1}+c_jb^{a_j},
\)
the effect of the new term \(c_jb^{a_j}\) on the valuation of the partial
sum is determined by the size of \(ea_j\) relative to \(r_{j-1}\). We
distinguish two cases.

\smallskip
\noindent\emph{Case 1: \(ea_j>r_{j-1}\).}
Since \(c_j\) is a positive integer,
\[
v_p(c_jb^{a_j})
=
v_p(c_j)+ea_j
\ge ea_j
>
r_{j-1}
=
v_p(T_{j-1}).
\]
The two summands \(T_{j-1}\) and \(c_jb^{a_j}\) therefore have distinct
\(p\)-adic valuations. By Lemma~\ref{lem:valuation-control}(i),
\[
v_p(T_j)
=
v_p(T_{j-1}),
\]
and hence
\(
r_j=r_{j-1}\le R_{j-2}.
\)
Also, the sequence \((R_k)\) is nondecreasing. Indeed, since
\(\lambda\ge1\) and \(B(\nu)>1\),
\[
\lambda R_k+\log_p B(\nu)>R_k,
\]
so
\[
R_{k+1}\ge R_k.
\]
Consequently,
\[
r_j\le R_{j-2}\le R_{j-1}.
\]

\smallskip
\noindent\emph{Case 2: \(ea_j\le r_{j-1}\).}
In this case,
\[
a_j\le\frac{r_{j-1}}{e}.
\]
Since the exponents satisfy \(a_i\le a_j\) for \(1\le i\le j\), we have
\(
b^{a_i}\le b^{a_j}.
\)
Therefore
\[
T_j \le b^{a_j}\sum_{i=1}^{q+1}c_i=B(\nu)b^{a_j}.
\]
Since \(a_j\le r_{j-1}/e\), we obtain
\[
0<T_j
<
B(\nu)b^{r_{j-1}/e}.
\]
Now \(p^{r_j}\mid T_j\), so
\(
p^{r_j}\le T_j.
\)
It follows that
\[
p^{r_j}
<
B(\nu)b^{r_{j-1}/e}.
\]
Taking logarithms to base \(p\) gives
\[
r_j
<
\log_p B(\nu)
+\frac{\log_p b}{e}\,r_{j-1}.
\]
By the definition of \(\lambda\),
\[
r_j
<
\log_p B(\nu)+\lambda r_{j-1}.
\]
Using the induction hypothesis \(r_{j-1}\le R_{j-2}\), we get
\[
r_j
<
\log_p B(\nu)+\lambda R_{j-2}.
\]
Since \(r_j\) is an integer,
\[
r_j
\le
\left\lfloor
\lambda R_{j-2}+\log_p B(\nu)
\right\rfloor
=
R_{j-1}.
\]
Thus \eqref{eq:induction-claim} holds in both cases. By induction,
\[
r_j\le R_{j-1}
\qquad\text{for every }1\le j\le q+1.
\]

Taking \(j=q+1\), and recalling that \(T_{q+1}=(b-1)x\), we obtain
\[
v_p((b-1)x)=r_{q+1}\le R_q.
\]
Because \(p\mid b\), we have \(p\nmid b-1\), and therefore
\[
v_p(b-1)=0.
\]
Hence
\[
v_p(x)
=
v_p((b-1)x)
\le R_q.
\]

It remains to obtain a closed expression for \(R_q\). From
\eqref{eq:recursive-R},
\[
R_{k+1}
=
\left\lfloor
\lambda R_k+\log_p B(\nu)
\right\rfloor
\le
\lambda R_k+\log_p B(\nu).
\]
Iterating this inequality gives
\[
\begin{aligned}
R_q
&\le
\lambda^qR_0+
\log_p B(\nu)
\left(1+\lambda+\cdots+\lambda^{q-1}\right)\\
&=
\rho_p(\nu)\lambda^q
+
\log_p B(\nu)
\sum_{j=0}^{q-1}\lambda^j.
\end{aligned}
\]
If \(\lambda>1\), then
\[
\sum_{j=0}^{q-1}\lambda^j
=
\frac{\lambda^q-1}{\lambda-1},
\]
and therefore
\[
v_p(x)
\le
\left\lfloor
\rho_p(\nu)\lambda^q
+
\log_p B(\nu)
\frac{\lambda^q-1}{\lambda-1}
\right\rfloor,
\]
which is \eqref{eq:general-bound}.

If \(\lambda=1\), then
\[
\sum_{j=0}^{q-1}\lambda^j=q,
\]
and consequently
\[
v_p(x)
\le
\left\lfloor
\rho_p(\nu)+q\log_p B(\nu)
\right\rfloor,
\]
which is \eqref{eq:general-bound-pure-power}.
\end{proof}
We now specialize Theorem~\ref{thm:general-bound} to the decimal case
\(b=10\) with \(p=2\). In this setting, we write
\begin{equation} \label{eq3}
q=\sum_{d=2}^9 \nu_d, \ \
B=1+9\sum_{d=2}^9(d-1)\nu_d, \ \
c=\log_2 10.
\end{equation}

If the units digit is even, then it must be one of \(2,4,6,\) or \(8\). The
corresponding initial values appearing in the definition of \(\rho_p(\nu)\)
are

\[
\begin{array}{c@{\quad}c@{\quad}c}
\toprule
d & 9(d-1)-1 & v_2\bigl(9(d-1)-1\bigr)\\
\midrule
2 & 8  & 3\\
4 & 26 & 1\\
6 & 44 & 2\\
8 & 62 & 1\\
\bottomrule
\end{array}
\]
and therefore
\begin{equation}\label{eq 5}
\rho=
\begin{cases}
3, & \nu_2>0,\\
2, & \nu_2=0 \text{ and } \nu_6>0,\\
1, & \nu_2=\nu_6=0 \text{ and } \nu_4+\nu_8>0,\\
0, & \nu_2=\nu_4=\nu_6=\nu_8=0.
\end{cases}
\end{equation}

Applying Theorem~\ref{thm:general-bound} with these quantities gives the
following Corollary.
\begin{corollary}
\label{cor:brier-conjecture}
Fix $\nu = (\nu_2,\ldots,\nu_9)\in\Nzero^8$. We have
\[
 \sup_{x\in\A_{10}(\nu)}v_2(x)<\infty.
\]
More precisely, every $x\in\A_{10}(\nu)$ satisfies
\begin{equation}\label{eq:decimal-closed-bound}
 v_2(x)\le
 \left\lfloor
 \rho c^q+\frac{c^q-1}{c-1}\log_2 B
 \right\rfloor,
\end{equation}
where $q,B,c,$ and $\rho$ are given by \eqref{eq3} and \eqref{eq 5}.
\end{corollary}

Corollary~\ref{cor:brier-conjecture} proves that the \(2\)-adic valuation is
uniformly bounded when the exceptional digits are fixed. However, the bound
obtained need not to be the smallest
possible one. In the next section, we show that the exact maximal valuation can
be computed by a finite procedure.

\section{The optimal bound for the \(p\)-adic valuation}\label{sec:sharp}

Fix an integer \(b\geq3\), a prime \(p\mid b\), and
\(\nu\in\Nzero^{b-2}\), and set
\(
e:=v_p(b).
\)
Let \(\mathcal D\) be the prescribed multiset of exceptional digits, containing
exactly \(\nu_d\) copies of \(d\) for each \(2\leq d\leq b-1\), and put
\[
q:=\sum_{d=2}^{b-1}\nu_d.
\]
Theorem~\ref{thm:general-bound} shows that the set
\(
\left\{v_p(x):x\in\A_b(\nu)\right\}
\)
is bounded and therefore has a maximum. Computing this maximum directly is
not immediate because \(\A_b(\nu)\) is infinite: arbitrarily many digits equal
to \(1\) may be inserted, so the length of \(x\) and the positions of its
exceptional digits are unbounded. We overcome this difficulty by following
the partial sum obtained as the exceptional digits are placed successively.

Let \(x\in\A_b(\nu)\) satisfying \(p\mid x\). Since \(p\mid b\), divisibility by
\(p\) is determined by the units digit. Thus, the units digit must be an
exceptional digit divisible by \(p\). List the exceptional digits of \(x\) as
\(d_1,\ldots,d_q\), in increasing order of their positions given by:
\[
0=r_1<r_2<\cdots<r_q<L.
\]
By Lemma~\ref{eq:repunit-correction},
\[
(b-1)x=-1+b^L+\sum_{i=1}^{q}(b-1)(d_i-1)b^{r_i}.
\]
Since \(p\nmid b-1\), we have
\[
v_p(x)=v_p\bigl((b-1)x\bigr).
\]

For \(1\leq m\leq q\), define
\[
T_m:=-1+\sum_{i=1}^{m}(b-1)(d_i-1)b^{r_i}.
\]
Thus, \(T_m\) is the partial sum obtained after the first \(m\) exceptional
digits have been placed, and
\[
(b-1)x=T_q+b^L.
\]
For \(1\leq m<q\), the next digit \(d_{m+1}\), placed at position
\(r_{m+1}>r_m\), gives
\[
T_{m+1}
=
T_m+(b-1)(d_{m+1}-1)b^{r_{m+1}}.
\]

The new term satisfies
\[
v_p\bigl((b-1)(d_{m+1}-1)b^{r_{m+1}}\bigr)
=
v_p(d_{m+1}-1)+er_{m+1}
\geq er_{m+1}.
\]
If \(er_{m+1}>v_p(T_m)\), Lemma~\ref{lem:valuation-control}(i) gives
\[
v_p(T_{m+1})=v_p(T_m).
\]
Moreover, all exceptional digits placed afterward, as well as the leading
term \(b^L\), occur at larger positions and therefore cannot change this
valuation. Consequently, after the first \(m\) digits have been placed, the
next position needs to be examined only if
\begin{equation}\label{eq:relevant-position}
r_{m+1}\leq
\left\lfloor\frac{v_p(T_m)}{e}\right\rfloor.
\end{equation}
This is the observation that makes the search finite.

We now express the same construction without fixing a particular integer
\(x\). A state of the search is a triple
\[
(\mathcal M,a,T),
\]
where \(\mathcal M\) is the multiset of exceptional digits not yet placed,
\(a\) is the largest occupied position, and \(T\) is the current partial sum.
For the integer considered above, the state after the first \(m\) exceptional
digits have been placed is
\[
\left(
\mathcal D\setminus\{d_1,\ldots,d_m\},
r_m,
T_m
\right),
\]
where one occurrence of each listed digit is removed.

Let
\[
\mathfrak V(\mathcal M,a,T)
\]
denote the largest final \(p\)-adic valuation obtainable from the state
\((\mathcal M,a,T)\). If \(\mathcal M\neq\varnothing\), place a digit
\(d\in\mathcal M\) at a position \(t>a\) and set
\[
T_{d,t}:=T+(b-1)(d-1)b^t.
\]
By \eqref{eq:relevant-position}, only positions satisfying
\[
a<t\leq\left\lfloor\frac{v_p(T)}{e}\right\rfloor
\]
can change the current valuation. Therefore,
\begin{equation}\label{eq:recursive-sharp}
\begin{aligned}
\mathfrak V(\mathcal M,a, T)
=\max\Biggl(
\{v_p(T)\}\cup \left\{
\mathfrak V\bigl(\mathcal M\setminus\{d\},t,T_{d,t}\bigr):
\begin{array}{l}
d\in\mathcal M, \ \
a<t\leq\left\lfloor v_p(T)/e\right\rfloor
\end{array}
\right\}
\Biggr).
\end{aligned}
\end{equation}
Here \(\mathcal M\setminus\{d\}\) means that one occurrence of \(d\) is
removed. The value \(v_p(T)\) represents the case in which the next digit is
placed beyond the relevant range. In that case, neither this digit nor any
subsequent term changes the valuation.

When \(\mathcal M=\varnothing\), all exceptional digits have been placed and
only the leading term \(b^L\) remains. Hence
\begin{equation}\label{eq:recursive-terminal}
\begin{aligned}
\mathfrak V(\varnothing,a,T)
=\max\Biggl(
\{v_p(T)\}\cup \left\{
v_p(T+b^L):
a<L\leq\left\lfloor\frac{v_p(T)}{e}\right\rfloor
\right\}
\Biggr).
\end{aligned}
\end{equation}
If \(eL>v_p(T)\), then \(b^L\) cannot change the valuation, which accounts for
the first value in the maximum.

The recursion is finite and exact. At each nonterminal state, one exceptional
digit is removed from \(\mathcal M\), and only finitely many positions satisfy
\eqref{eq:relevant-position}. Moreover, every completion of a state either
places the next digit within this finite range, in which case it appears in
the recursive part of \eqref{eq:recursive-sharp}, or places it beyond the
range, in which case the final valuation remains \(v_p(T)\).

It remains to identify the possible initial states. If \(p\nmid x\), then
\(v_p(x)=0\). If \(p\mid x\), its units digit is some
\(d\in\mathcal D\) satisfying \(p\mid d\). Placing \(d\) in position \(0\)
gives
\[
T_d:=(b-1)(d-1)-1
\]
and the initial state
\(
\bigl(\mathcal D\setminus\{d\},0,T_d\bigr).
\)
We have therefore established the following exact formula.

\begin{theorem}\label{thm:exact-search}
Let \(b\geq3\), let \(p\mid b\) be a prime, and let
\(
\nu=(\nu_2,\ldots,\nu_{b-1})\in\Nzero^{b-2}.
\)
Let \(\mathcal D\) be the multiset containing exactly
\(\nu_d\) copies of each digit \(d\in\{2,\ldots,b-1\}\), and, for
\(d\in\mathcal D\), define
\[
T_d:=(b-1)(d-1)-1.
\]
Then the recursion defining \(\mathfrak V\) terminates after finitely many
steps, and
\begin{equation}\label{eq:exact-max}
\begin{aligned}
\max_{x\in\A_b(\nu)}v_p(x)
=\max\Biggl(
\{0\}\cup \left\{
\mathfrak V\bigl(\mathcal D\setminus\{d\},0,T_d\bigr):
\begin{array}{l}
d\in\mathcal D, \ \
p\mid d
\end{array}
\right\}
\Biggr).
\end{aligned}
\end{equation}

\end{theorem}

We illustrate the procedure in base \(10\) with \(p=2\) and
\[
\mathcal D=\{4,4,7\}.
\]
Here \(e=v_2(10)=1\). An even integer in this family must have units digit
\(4\). Placing one copy of \(4\) in position \(0\) gives
\[
T_4=9(4-1)-1=26,
\qquad
v_2(T_4)=1,
\]
so the initial state is
\[
(\{4,7\},0,26).
\]

The next position must satisfy
\[
0<t\leq1,
\]
so \(t=1\) is the only possible position. There are, however, two possible
digits to place there. They give
\[
\begin{array}{c@{\qquad}c@{\qquad}c}
\toprule
\text{digit }d & 26+9(d-1)10 & \text{\(2\)-adic valuation}\\
\midrule
4 & 296 & 3\\
7 & 566 & 1\\
\bottomrule
\end{array}
\]
We therefore place the second digit \(4\) in position \(1\), producing the state
\[
(\{7\},1,296),
\qquad
v_2(296)=3.
\]

For the remaining digit \(7\), the relevant positions satisfy
\[
1<t\leq3,
\]
so both \(t=2\) and \(t=3\) must be examined:
\[
\begin{array}{c@{\qquad}c@{\qquad}c}
\toprule
t & 296+54\cdot10^t & \text{\(2\)-adic valuation}\\
\midrule
2 & 5696  & 6\\
3 & 54296 & 3\\
\bottomrule
\end{array}
\]
For \(t=2\), we obtain the terminal state
\[
(\varnothing,2,5696),
\qquad
v_2(5696)=6.
\]
The leading position must now satisfy \(L>2\), and only the values
\(L=3,4,5,6\) can change the current valuation. Checking all of them gives
\[
\begin{array}{c@{\qquad}c@{\qquad}c}
\toprule
L & 5696+10^L & \text{\(2\)-adic valuation}\\
\midrule
3 & 6696    & 3\\
4 & 15696   & 4\\
5 & 105696  & 5\\
6 & 1005696 & 7\\
\bottomrule
\end{array}
\]
Thus, the largest value
among all terminal choices is \(7\), attained when \(L=6\). Since \(9\) is
odd,
\[
v_2(x)=v_2(9x)=7.
\]
The corresponding integer is
\[
x=\frac{1005696}{9}=111744,
\]
and indeed
\[
111744=3^2\cdot97\cdot2^7.
\]

Applying the same finite procedure to the four exceptional-digit multisets
considered in \cite[Table~4]{BrierEtAl2021} gives the following exact values.
A maximizing integer is recovered by recording the choices made along a
maximizing branch of the recursion.

\begin{center}
\begin{tabular}{ccl}
\toprule
multiset & \(\max v_2\) & one maximizing integer\\
\midrule
\(\{2,2,2,2,7\}\) & \(13\) &
\(172122112=21011\cdot2^{13}\)\\
\(\{2,2,4,7\}\) & \(15\) &
\(211111411712=17\cdot378977\cdot2^{15}\)\\
\(\{4,4,7\}\) & \(7\) &
\(111744=3^2\cdot97\cdot2^7\)\\
\(\{2,7,8\}\) & \(9\) &
\(1178112=3\cdot13\cdot59\cdot2^9\)\\
\bottomrule
\end{tabular}
\end{center}

Thus, Corollary~\ref{cor:brier-conjecture} gives a convenient explicit bound
valid over the whole family, whereas Theorem~\ref{thm:exact-search} gives a
finite procedure for computing the optimal bound for each prescribed
multiset of exceptional digits.

\section{Consequences for nonzero even terminal digits}
\label{sec:even-program}

We now explain how our valuation bounds could contribute to the treatment of the
nonzero even terminal digits
\(
\delta\in\{2,4,6,8\}.
\)

Throughout this section, \(S=S_{10}\) denotes the decimal digit-product map.

For a fixed terminal digit \(\delta\), let
\[
A_\delta=(V_\delta,E_\delta)
\]
denote its full backward graph. More precisely, \(A_\delta\) is the smallest
directed graph constructed as follows. We begin with \(\delta\in V_\delta\).
Whenever \(s\in V_\delta\) and \(x\) is a positive integer satisfying
\(
S^2(x)=s,
\)
we add the intermediate value
\(
y:=S(x)
\)
to \(V_\delta\), together with the directed edge
\(
s\longrightarrow y.
\)
Thus, an arrow \(s\to y\) means that \(S(y)=s\) and that \(y\) itself occurs
as the digit product of another integer. The arrows point away from the
terminal digit and toward its possible predecessors.

If \(A_\delta\) was finite, then the length of its longest path would give a uniform
bound on the persistence of all such integers. Motivated by this observation,
Brier, Clavier, Gutsche, and Naccache constructed finite candidate graphs
\[
B_\delta=(U_\delta,F_\delta)
\]
for the nonzero even terminal digits \cite[
Appendix~C]{BrierEtAl2021}. Proving that \(B_\delta=A_\delta\) would show that
these graphs contain all possible predecessors and would therefore establish
the corresponding persistence bounds. Here \(U_\delta\) is the proposed finite set of
vertices and \(F_\delta\) is the proposed set of directed edges between them.
By construction,
\[
U_\delta\subseteq V_\delta,
\qquad
F_\delta\subseteq E_\delta.
\]
The remaining problem is to prove that no additional vertices are missing.
Equivalently, one must verify that
\begin{equation}\label{eq:backward-closure}
s\in U_\delta,\quad S^2(x)=s
\quad\Longrightarrow\quad
S(x)\in U_\delta.
\end{equation}
If \eqref{eq:backward-closure} holds for every \(s\in U_\delta\), then
\(
B_\delta=A_\delta.
\)

For example, the candidate finite graph for the terminal digit \(4\) proposed in \cite{BrierEtAl2021} has vertex set
\[
U_4=
\left\{
4,14,27,72,98,189,294,1161216,
2^{23}3^7\cdot7
\right\}.
\]
Its edges are displayed in Figure 1. The loop at \(4\)
records \(S(4)=4\), while, for example,
\[
S(14)=4,\qquad
S(72)=14,\quad \text{and} \quad
S(1161216)=72.
\]

\begin{figure}[h] \label{fig:B4-graph}
\centering
\begin{tikzpicture}[
    vertex/.style={
        draw,
        rounded corners=2pt,
        minimum height=7mm,
        inner xsep=7pt,
        font=\small
    },
    arrow/.style={
        -{Latex[length=2mm,width=1.4mm]},
        semithick
    }
]
\node[vertex] (v4) at (0,0) {\(4\)};
\node[vertex] (v14) at (2,0) {\(14\)};
\node[vertex] (v27) at (4,1) {\(27\)};
\node[vertex] (v72) at (4,-1) {\(72\)};
\node[vertex] (v98) at (6.4,1.2) {\(98\)};
\node[vertex] (v189) at (6.4,0.2) {\(189\)};
\node[vertex] (v294) at (6.4,-0.8) {\(294\)};
\node[vertex] (v1161216) at (6.4,-1.8) {\(1161216\)};
\node[vertex] (vlast) at (10,-1.8) {\(2^{23}3^7\cdot7\)};

\draw[arrow] (v4) edge[loop above,looseness=5] (v4);
\draw[arrow] (v4) -- (v14);
\draw[arrow] (v14) -- (v27);
\draw[arrow] (v14) -- (v72);
\draw[arrow] (v72) -- (v98);
\draw[arrow] (v72) -- (v189);
\draw[arrow] (v72) -- (v294);
\draw[arrow] (v72) -- (v1161216);
\draw[arrow] (v1161216) -- (vlast);
\end{tikzpicture}
\caption{The candidate backward graph \(B_4=(U_4,F_4)\).}
\end{figure}
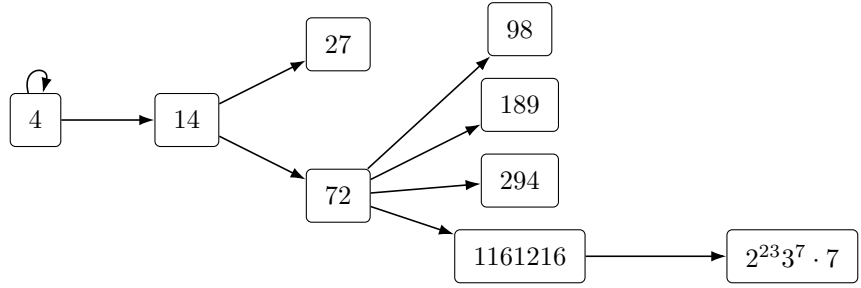

We next describe how the completeness condition
\eqref{eq:backward-closure} is verified. Fix \(s\in U_\delta\), and suppose
that
\[
S^2(x)=s.
\]
Set
\[
y:=S(x).
\]
Then \(S(y)=s\). Hence the product of the decimal digits of \(y\) is \(s\).
Every possible multiset of non-\(1\) digits of \(y\) is therefore obtained
from a factorization
\[
s=\prod_{d=2}^{9}d^{n_d},
\qquad
n_d\in\Nzero.
\]
We denote the set of all such factorizations by
\[
\mathcal F(s):=
\left\{
(n_2,\ldots,n_9)\in\Nzero^8:
\prod_{d=2}^{9}d^{n_d}=s
\right\}.
\]
For each
\[
\mathbf n=(n_2,\ldots,n_9)\in\mathcal F(s),
\]
let \(\mathcal D(\mathbf n)\) be the multiset containing \(n_d\) copies of
the digit \(d\) with \( 2 \le d \le 9\). The decimal expansion of \(y\) then consists of the digits
in \(\mathcal D(\mathbf n)\), together with an arbitrary number of digits
equal to \(1\).

Write
\[
\mathcal D(\mathbf n)=\{d_1,\ldots,d_k\},
\qquad
k=\sum_{d=2}^{9}n_d.
\]
Let \(a_0\) be the number of digits of \(y\), and let \(a_i\) be the position
of the exceptional digit \(d_i\). Thus,
\begin{equation}\label{eq:even-position-conditions}
0\leq a_i<a_0
\quad (1\leq i\leq k),
\qquad
a_i\neq a_j
\quad (i\neq j).
\end{equation}
Lemma~\ref{eq:repunit-correction} gives
\[
9y=
10^{a_0}-1+
\sum_{i=1}^{k}9(d_i-1)10^{a_i}.
\]

Because \(y=S(x)\), every prime divisor of \(y\) belongs to
\(\{2,3,5,7\}\). Moreover, an orbit ending at one of
\(2,4,6,\) or \(8\) cannot pass through a value containing a factor \(5\):
such a value leads either to the terminal digit \(5\) or to \(0\).
Consequently, the relevant intermediate values have the form
\[
y=2^t3^u7^w,
\qquad
t,u,w\in\Nzero.
\]
Therefore, the possible predecessors associated with
\(\mathcal D\) are exactly the solutions of
\begin{equation}\label{eq:general-even-equation}
10^{a_0}-1+
\sum_{i=1}^{k}9(d_i-1)10^{a_i}
=
2^t3^{u+2}7^w
\end{equation}
satisfying the positional conditions
\eqref{eq:even-position-conditions}. The exponent \(t\) is precisely
\(v_2(y)\).

Theorem~\ref{thm:exact-search} computes
\[
t_{\max}(\mathcal D)
:=
\max_{y\in\A_{10}(\nu)}v_2(y).
\]
Hence every solution of \eqref{eq:general-even-equation} satisfies
\[
0\leq t\leq t_{\max}(\mathcal D).
\]

Using the modular method of
\cite{BrierEtAl2021} gives the following finite procedure:
\begin{enumerate}[label=\textup{(\arabic*)},leftmargin=*]
\item For each \(s\in U_\delta\), enumerate the finite set
\(\mathcal F(s)\) of digit factorizations of \(s\).

\item For each \(\mathbf n\in\mathcal F(s)\), form the corresponding multiset
\(\mathcal D(\mathbf n)\) and the equation
\eqref{eq:general-even-equation}.

\item Apply Theorem~\ref{thm:exact-search} to compute
\(t_{\max}(\mathcal D)\), and consider only
\(0\leq t\leq t_{\max}(\mathcal D)\).

\item For each fixed \(t\), apply the modular-lifting algorithm of
\cite{BrierEtAl2021}.

\item Discard every solution that violates
\eqref{eq:even-position-conditions}. Each surviving solution determines a
possible predecessor \(y\). Verify that every such \(y\) belongs to
\(U_\delta\).
\end{enumerate}
If the last verification succeeds for every \(s\in U_\delta\), then
\eqref{eq:backward-closure} holds and the candidate graph is complete.

The size of this computation was estimated in
\cite{BrierEtAl2021}:
\[
\begin{array}{c@{\qquad}c@{\qquad}c@{\qquad}c}
\toprule
\delta & \lvert U_\delta\rvert
& \text{equation families}
& \text{largest number of terms}\\
\midrule
2 & 33 & 1117 & 30\\
4 & 9  & 1062 & 32\\
6 & 84 & 6377 & 37\\
8 & 51 & 4774 & 45\\
\bottomrule
\end{array}
\]
The numbers in the third column count the digit factorizations before the
finite range of \(t\) is taken into account. Including all admissible values of
\(t\) therefore produces a larger, but still finite, collection of equations.
Their exhaustive analysis is well suited to exact high-performance
computational methods, and its completion would determine whether the proposed
backward graphs are complete. The purpose of the present work is to provide the
missing theoretical bound and to reduce this verification to a finite,
terminating procedure. Although carrying out the full computation lies beyond
the scope of this paper, the formulation developed here is intended to make
that next step accessible.


\begin{thebibliography}{99}

\bibitem{BonuccelliColucciFaria2020}
G. Bonuccelli, L. Colucci and E. de Faria,
\emph{On the Erd\H{o}s--Sloane and shifted Sloane persistence problems},
J. Integer Seq. \textbf{23} (2020), Article 20.10.7, 30 pp.

\bibitem{BrierEtAl2021}
\'E. Brier, C. Clavier, L. Gutsche and D. Naccache,
\emph{The multiplicative persistence conjecture is true for odd targets},
arXiv:2110.04263 [math.NT], 2021,
\url{https://arxiv.org/abs/2110.04263}.

\bibitem{deFariaTresser2014}
E. de Faria and C. Tresser,
\emph{On Sloane's persistence problem},
Exp. Math. \textbf{23} (2014), no.~4, 363--382.


\bibitem{Guy2004}
R. K. Guy,
\emph{Unsolved Problems in Number Theory},
3rd ed., Problem Books in Mathematics, Springer, New York, 2004.

\bibitem{LamontSmith2021}
T. Lamont-Smith,
\emph{Multiplicative persistence and absolute multiplicative persistence},
J. Integer Seq. \textbf{24} (2021), Article 21.6.7, 26 pp.


\bibitem{PerezStyer2015}
S. Perez and R. Styer,
\emph{Persistence: a digit problem},
Involve \textbf{8} (2015), no.~3, 439--446.

\bibitem{Sloane1973}
N. J. A. Sloane,
\emph{The persistence of a number},
J. Recreational Math. \textbf{6} (1973), 97--98.


\end{thebibliography}
\end{document}